\documentclass[12pt]{article}
\usepackage[utf8]{inputenc}
\usepackage{amsfonts}
\usepackage{amssymb}
\usepackage{amsthm}
\usepackage[fleqn]{amsmath}
\usepackage{epsf, epsfig}
\usepackage{amssymb,latexsym}
\usepackage{color}

\usepackage{subfigure}
\usepackage{graphicx}
\usepackage{lipsum}
\usepackage{pinlabel}

\newtheorem{D}{Definition}%[section]
\newtheorem{T}{Theorem}%[section]
\newtheorem{C}{Corollary}%[section]
\newtheorem{Conj}{Conjecture}%[section]
\newtheorem{Pro}{Problem}%[section]
\newtheorem{prop}{Proposition}%[section]

\newcommand {\relabel}[1] {\label{#1} \red{[#1]}} \newcommand {\rebibitem}[1] {\bibitem{#1} \red{[#1]}}%draft

\def\relabel {\label} \def\rebibitem {\bibitem}%chosen for final

\begin{document}

\title{
%Do hypergraphs have properties on chromatic polynomials not owned by graphs
Properties of chromatic polynomials of hypergraphs not held for chromatic polynomials of graphs\thanks{This paper was partially supported by NTU AcRf Project RP 3/16 DFM of Singapore.}
%Do the chromatic polynomials of graphs and the chromatic polynomials of hypergraphshave different properties? 
}
\author{
Ruixue Zhang, 
Fengming Dong\thanks{Corresponding author.
Email: fengming.dong@nie.edu.sg}\\
\small Mathematics and Mathematics Education\\
\small National Institute of Education\\
\small Nanyang Technological University, Singapore 637616\\
}
\date{}

\maketitle

\def \calv {\mathcal{V}}
\def \cali {\mathcal{I}}
\def \calh {\mathcal{H}}
\def \cale {\mathcal{E}}
\def \sets {\mathcal{S}} %{{\cal S}}

\def \calho {\mathcal{H}_{\bullet G}}

\newcommand \calhd[1] {\calh_{\bullet #1}}

\begin{abstract}
In this paper, we present some properties on 
chromatic polynomials of hypergraphs 
which do not hold for chromatic polynomials of graphs.
We first show that chromatic polynomials of hypergraphs 
have all integers as their zeros
and contain dense real zeros in the set of real numbers.
We then prove that for any multigraph $G=(V,E)$, 
the number of totally cyclic orientations of $G$ is equal to 
the value of $|P(\calh_G,-1)|$, where $P(\calh_G,\lambda)$  
is the chromatic polynomial of a hypergraph $\calh_G$ 
which is constructed from $G$.
Finally we show that 
the multiplicity of root ``$0$" of $P(\calh,\lambda)$ 
may be at least $2$ for some connected hypergraphs $\calh$,
and the multiplicity of root ``$1$" of $P(\calh,\lambda)$ 
may be $1$ for some connected and separable hypergraphs $\calh$
and may be $2$ for some connected and non-separable 
hypergraphs $\calh$.
\end{abstract}

\section{Introduction and main results}

For any graph $G=(V,E)$, the {\it chromatic polynomial} of $G$ 
is the function $P(G,\lambda)$ such that
for any positive integer $\lambda$,  
$P(G,\lambda)$ is the number of 
proper $\lambda$-colourings
of $G$,
where a proper $\lambda$-colouring of $G$ is a mapping 
$\phi:V\rightarrow \{1,2,\cdots,\lambda\}$ 
such that $\phi(u)\ne \phi(v)$ holds for each pair of 
adjacent vertices $u$ and $v$ in $G$.
This graph-function $P(G,\lambda)$ was originally introduced by 
Birkhoff~\cite{bir1912} in 1912 in the hope of proving 
the four-color theorem
(i.e., $P(G,4)>0$ holds for any loopless planar graph $G$).

A {\it hypergraph $\calh$} consists of an order pair of vertex set $\mathcal{V}$ and edge set $\cale$,
where $\cale$ is a subset of 
$\{e\subseteq \calv: |e|\ge 1\}$.
If $|e|\le 2$ for all $e\in \cale$, then 
$\calh$ is a graph. 
For any integer $\lambda\geq 1$, 
a \textit{weak proper $\lambda$-colouring} of 
a hypergraph $\calh=(\calv, \cale)$ 
is a mapping $\phi: \calv \rightarrow \{1,2,\cdots,\lambda\}$ %(not necessarily onto) 
such that $|\{\phi(v): v\in e\}|>1$ holds for each $e\in \cale$
(see \cite{Jones 1976, Jones 1976b, Vol 2002}). 
Thus $\calh$ does not have any weak proper $\lambda$-colouring
if $|e|=1$ for some edge $e\in \cale$.
A \textit{strong proper $\lambda$-colouring} of 
$\calh=(\calv, \cale)$ 
is a mapping $\phi: \calv \rightarrow \{1,2,\cdots,\lambda\}$ %(not necessarily onto) 
such that $|\{\phi(v): v\in e\}|=|e|$ 
holds for each $e\in \cale$
(see \cite{Jones 1976, Jones 1976b,  Vol 2002} also). 

Note that %if $\calh$ is a hypergraph with $|e|\ge 2$ for each edge $e$, 
the number of distinct 
strong proper $\lambda$-colourings 
in $\calh=(\calv, \cale)$ is equal to 
$P(G,\lambda)$,
%(i.e., the number of proper $\lambda$-colourings in a graph $G$, 
where $G$ is the simple graph with vertex set $\calv$ 
in which any two distinct vertices $u,v\in \calv$ are adjacent 
if and only if 
they are contained in an edge $e\in \cale$ in $\calh$
(i.e., $G$ is obtained from $\calh$ by changing 
each edge in $\calh$ to a clique in $G$).
Thus the function counting the number 
of strong proper $\lambda$-colourings in $\calh$ 
is not different from the chromatic polynomial of a graph.
It is probably for this reason that 
most articles on chromatic polynomials of hypergraphs 
in the past two decades focused on the function 
counting the number of weak proper $\lambda$-colourings in 
a hypergraph $\calh$
(see \cite{Allagan 2007, Allagan 2011, Allagan 2014, 
Borowiecki 2000, Dohmen 1995,
Tomescu 1998, Tomescu 2004, Tomescu 2009, Tomescu 2014,
Vol 2002, Walter 2009}).

Let $P(\calh,\lambda)$ be the number of 
weak proper $\lambda$-colourings of $\calh$. 
It is obvious that this graph-function $P(\calh,\lambda)$
is an extension of the chromatic polynomial of a graph. 
%Observe that, if $\calh$ is a graph (i.e., $|e|=2$ for each $e\in \cale$), then $P(\calh,\lambda)$ is the chromatic polynomial of a graph (i.e., $\calh$).
In this paper, $P(\calh,\lambda)$ is called 
\textit{the chromatic polynomial of $\calh$}, and it is indeed a polynomial in $\lambda$ of degree $|\calv|$. 
This graph-function 
$P(\calh,\lambda)$ %for a hypergraph $\calh$ 
appeared in the work of Helgason \cite{H} in 1972, and 
it is unknown if it had been introduced earlier.
It  may have been noticed to be a polynomial by 
Chv\'atal~\cite{chv 1970}.
It has been studied extensively in the past twenty years 
by many researchers, such as 
Allagan \cite{Allagan 2007, 
Allagan 2011, Allagan 2014}, 
Borowiecki and \L{}azuka~\cite{Borowiecki 2000},
Dohmen~\cite{Dohmen 1995},
Tomescu~\cite{Tomescu 1998, Tomescu 2004,
Tomescu 2009, Tomescu 2014},
Voloshin~\cite{Vol 2002} and 
Walter \cite{Walter 2009}.
They extended many properties of 
chromatic polynomials of graphs 
on computations, expressions, factorizations, etc, 
to chromatic polynomials of hypergraphs.

The following are some known 
properties on chromatic polynomials
of graphs which 
also hold for chromatic polynomials of hypergraphs:
\begin{enumerate}
\renewcommand{\theenumi}{\rm (A.\arabic{enumi})}
\item Deletion/Contraction property 
\cite{rea1, rea2, Jones 1976, whi1932b}, 
cited as Theorem 5;

\item (by definition) multiplicativity with respect to disjoint unions, cited as Proposition 2;

\item  
the factorization formula 
for the chromatic polynomial of a graph with a cut-set 
which induced a clique
\cite{Borowiecki 2000, Jones 1976, zyk}, cited as Theorem 6; 

\item 
the roots of chromatic polynomials of graphs 
are dense in the complex plane
\cite{sokal 2004}; 
%as proved for graphs;
\item   
(as the family of hypergraphs include 
all graphs) 
the real roots of chromatic polynomials of graphs  
are dense in the interval $[32/27,\infty)$
(see \cite{jac1, tho2});

\item  
Whitney's Broken-cycle Theorem~\cite{whi1932b}, 
extended by Dohmen~\cite{Dohmen 1995}
to the set of linear hypergraphs which do not contain
any edge with an odd size and  
contain edges of size $2$ in every cycle;

\item  
for fixed $\lambda$, computing $P(H,\lambda)$ 
is polynomial time computable for classes
of graphs (hypergraphs) of bounded tree width. 
The proof also works for hypergraphs
\cite{Makowsky 2006};

\item 
for any graph $G$ of order $n$, we have 
$P(G,\lambda)=\lambda^n
+a_{n-1}\lambda^{n-1}+\cdots+a_1\lambda$,
where $a_{i}=\sum_{j\ge 0}(-1)^jN(i,j)$
and $N(i, j)$ denotes the number of 
spanning subgraphs of 
$G$ with $i$ components and $j$ edges
\cite{big1993, birk1946, D, rea1, rea2, whi1932b}.
Tomescu \cite{Tomescu 1998} 
showed that this result also holds 
chromatic polynomials of hypergraphs
(see \cite{Dohmen 1995} also); 

\item  
$0$ is a root of every chromatic polynomial 
and all positive integers are roots of chromatic polynomials
\cite{big1993, birk1946, Dohmen 1995, D, 
rea1, rea2, Tomescu 1998, whi1932b}.
\end{enumerate}

But chromatic polynomials of graphs also have the following
properties on its coefficients 
not held  for chromatic polynomials
of hypergraphs:
\begin{enumerate}
\renewcommand{\theenumi}{\rm (B.\arabic{enumi})}

\item 
for any graph $G$ of order $n$ and component number $c$,
if $P(G,\lambda)=\lambda^n
+a_{n-1}\lambda^{n-1}+\cdots+a_1\lambda$,
then $a_i\ne 0$  if and only if $c\le i\le n$
(see \cite{big1993, birk1946, D, rea1, rea2, whi1932b}). 
Instead, for a hypergraph $\calh=(\calv,\cale)$ of order $n$
with $|e|\ge 3$ for all $e\in \cale$,
(A.8) implies that 
if $P(\calh,\lambda)=\lambda^n
+b_{n-1}\lambda^{n-1}+\cdots+b_1\lambda$,
then  $b_{n-i}\ne 0$ 
but $b_{n-j}=0$ for all $j$ with $1\le j<i$, 
where $i=\min_{e\in \cale}|e|-1$
(see \cite{Dohmen 1995, Tomescu 1998});
\label{PN12}
\item  
the coefficients $1,a_{n-1},\cdots,a_c$ in 
the expansion of $P(G,\lambda)$ alternate in signs
(see \cite{big1993, birk1946,  D, rea1, rea2, whi1932b}).
Instead, for any linear hypergraph $\calh=(\calv,\cale)$ 
(i.e., $|e_1\cap e_2|\le 1$ for all distinct $e_1,e_2\in \cale$)
with $|\calv|=n$ and $h=\min_{e\in \cale}|e|\ge 3$, 
if $\calh$ contains edges of size $h+1$
and $P(\calh,\lambda)=\lambda^n
+b_{n-1}\lambda^{n-1}+\cdots+b_1\lambda$,
%it is possible that $b_i=0$ while $b_{i+a}\ne 0$ and $b_{i-b}\ne 0$ for some $i,a,b>0$.
then (A.8) implies that $b_{n-h+1}$ and $b_{n-h}$ 
have the same sign
(see \cite{Dohmen 1995, Tomescu 1998});

\item 
the sequence $1,|a_{n-1}|,\cdots,|a_2|, |a_1|$
is log-concave
(i.e., $a_j^2\ge |a_{j-1}|\cdot |a_{j+1}|$ holds for all $j$), where $a_i$'s are coefficients 
of terms of $P(G,\lambda)$ given in (B.1)
(see \cite{huh1, huh2, rea1, rea2}).
Instead, for a hypergraph $\calh$, 
if $P(\calh,\lambda)=\lambda^n
+b_{n-1}\lambda^{n-1}+\cdots+b_1\lambda$,
by (B.1),  it is possible that $b_{n-j}=0$ while 
$b_{n-k}\ne 0$ and $b_{n-i}\ne 0$ for some integers $k,j,i$
with $0\le k<j<i$
(see \cite{Dohmen 1995, Tomescu 1998}). 
\end{enumerate}

In this article, we will 
present some properties on chromatic polynomials of hypergraphs
which are different from the following properties on 
chromatic polynomials of graphs:
\begin{enumerate}
\renewcommand{\theenumi}{\rm (\roman{enumi})}
\item  
$(-\infty,0)$, $(0,1)$ and $(1,32/27]$ 
are zero-free intervals 
for chromatic polynomials of graphs \cite{jac1, rea2, tho2},
while chromatic polynomials of hypergraphs have dense real zeros 
in the whole set of real numbers, as stated in 
Corollary~\ref{mainco1} (a);
\item  for any graph $G$,
$|P(G, -1)|$ counts the number 
of acyclic orientations (i.e., orientations 
without any directed cycle) of a multigraph $G$ 
(see \cite{stanley}),
while $|P(\calh_G, -1)|$ counts the number 
of totally cyclic orientations of $G$
(i.e., orientations in which each arc is contained in 
some directed cycle),
as stated in Corollary~\ref{mainco2},
where  $\calh_G$ is obtained from $G$ by 
adding $|E|$ new vertices $\{w_{e}: e\in E\}$  
and changing each edge $e$ in $G$ 
to an edge $\{u_e,v_e,w_{e}\}$ in $\calh_G$,
where $u_e$ and $v_e$ are the two ends of $e$ in $G$;

\item 
the multiplicity of root `$0$' of 
$P(G,\lambda)$ counts the number of components of a graph $G$,
and thus $P(G,\lambda)$ has no factor $\lambda^2$ whenever 
$G$ is connected (see \cite{D, eis1972, rea1, rea2}),
while $P(\calh,\lambda)$ may have a factor $\lambda^2$ 
for a connected 
hypergraph $\calh$, as stated in Theorem~\ref{main3}; 
\item for a connected graph $G$, 
$P(G,\lambda)$ has a factor $(\lambda-1)^2$ 
if and only if $G$ is separable (see \cite{whi1984, woo1977}), while  $P(\calh,\lambda)$ may have no 
factor $(\lambda-1)^2$ for a connected and 
separable hypergraph $\calh$, 
as stated in Theorem~\ref{main4-0}.
\end{enumerate}

For any simple graph $G=(V,E)$ 
(i.e., $G$ has no loops nor parallel edges), 
let $\calho$ be the 
hypergraph with vertex set 
$\calv=V\cup \{w\}$ 
and edge set 
$\cale=\{\{u,v,w\}: uv\in E\}$, 
i.e., $\calho$ is obtained from $G$ 
by adding a new vertex $w$ 
and changing each edge $uv$ in $G$ 
to an edge $\{u,v,w\}$ in $\calho$.

{\it The independence polynomial} of a graph $G$
is defined to be $I(G,x)=\sum_{A} x^{|A|}$,
where the sum runs over all independent sets $A$ of $G$.
One of the main purposes in this article 
is to establish a relation between 
$P(\calho,\lambda)$ and $I(G,x)$
in the following theorem. 

\begin{T} \relabel{main1}
For any simple graph $G$ of order $n$, 
\begin{equation}\relabel{main1-eq1}
P(\calho,\lambda)=\lambda (\lambda-1)^n 
I(G,1/(\lambda-1)).
\end{equation}
\end{T}

Theorem~\ref{main1} implies 
that the multiplicity of root ``$1$" in $P(\calho,\lambda)$
is $n-\alpha(G)$, where $\alpha(G)$ is the 
independence number of $G$,
and whenever $z$ is a zero of $I(G,x)$, 
$1+1/z$ is a zero of $P(\calho,\lambda)$.
If $G$ is the complete graph $K_n$, then 
$I(G,x)=1+nx$ and by Theorem~\ref{main1}, 
$1-n$ is a zero of $P(\calho,\lambda)$.

Brown, Hickman and Nowakowski~\cite{B} 
showed that real roots of independence polynomials are dense in $(-\infty,0]$ while the complex roots 
of these polynomials
are dense in $\mathbb{C}$ (i.e. the whole complex plane). Chudnovsky and Seymour~\cite{M} proved that if $G$ is clawfree, then all the roots of its independence polynomial are real. 
By Theorem~\ref{main1} and results in \cite{B,M}, 
we have the following conclusions immediately
except Corollary~\ref{mainco1} (b) whose proof 
will be given in Section 3.

%Together with Corollary \ref{co4}, we obtain the followings.

\begin{C}\relabel{mainco1}
\begin{enumerate}
\item[(a)] The complex roots of $P(\calho,\lambda)$ 
for all graphs $G$ are dense in the whole complex plane;
\item[(b)] The real roots of $P(\calho,\lambda)$ 
for all graphs $G$ 
are dense in the set of real numbers;
\item[(c)] Every negative integer is a root of 
$P(\calho,\lambda)$  for some graph $G$;
\item[(d)] If $G$ is clawfree, then all roots of 
$P(\calho,\lambda)$ are real.
\end{enumerate}
\end{C}

For a multigraph $G=(V,E)$, 
let $\calh_G=(\calv,\cale)$ be another hypergraph 
constructed from $G$, 
where $\calv=V\cup \{w_e: e\in E\}$ and 
$\cale=\{\{u_e,v_e,w_e\}: e\in E\}$, 
where $u_e$ and $v_e$ are the two ends of $e$ in $G$,
i.e., $\calh_G$ is obtained from $G$ by 
adding $|E|$ new vertices $\{w_{e}: e\in E\}$  
and changing each edge $e$ in $G$ 
to an edge $\{u_e,v_e,w_{e}\}$ in $\calh_G$.
We will express 
$P(\calh_G,\lambda)$ by the Tutte polynomial $T_G(x,y)$ of $G=(V,E)$, where %$T_G(x,y)$ is defined below:
\begin{equation}\relabel{tuttep}
T_G(x,y)=\sum_{A\subseteq E}(x-1)^{r(E)-r(A)}(y-1)^{|A|-r(A)},
\end{equation}
$r(A)=|V|-c(A)$ and $c(A)$ is the number of components 
of the spanning subgraph $(V,A)$ of $G$
for any $A\subseteq E$.

\begin{T}\relabel{main2}
For any multigraph $G=(V,E)$ with order $n$ and size $m$, 
\begin{equation}\relabel{main2-eq1}
P(\calh_{G},\lambda)=\lambda^{m-n+2c(G)}
\cdot (-1)^{n+c(G)}\cdot 
T_G(1-\lambda^{2},(\lambda-1)/\lambda),
\end{equation}
where $c(G)$ is the number of components of $G$,
i.e., $c(G)=c(E)$.
\end{T}

Stanley~\cite{stanley} showed that for any multigraph $G$, 
$|P(G,-1)|=T_G(2,0)$
counts the number of acyclic orientations 
of $G$, where an acyclic orientation of $G$ 
is an orientation of $G$ such that 
the digraph obtained does not have any directed cycle. 
By Theorem~\ref{main2}, 
the number of totally cyclic orientations of $G$ 
(i.e., orientations of $G$ on which each arc is in some 
directed cycle)
can be determined by the value of $|P(\calh_{G},-1)|$.

\begin{C}\relabel{mainco2}
For any multigraph $G$, 
$|P(\calh_{G},-1)|=T_G(0,2)$
counts the number of totally cyclic orientations 
of $G$.
\end{C}

It is well known that for any connected graph $G$,
$P(G,\lambda)$ has a factor $\lambda$ but no 
factor $\lambda^2$ (see \cite{D, eis1972, rea1, rea2, whi1932b}).
However, 
$P(\calh,\lambda)$ may have a factor $\lambda^2$
for a connected hypergraph $\calh$, as stated 
in Theorem~\ref{main3}..

%has a factor $\lambda$ and even may have 

\iffalse 
It is known that for any hypergraph $\calh$, 
$P(\calh,\lambda)$ has a factor $\lambda$
but may also have a factor 
$\lambda^2$ when $\calh$ is connected, where 
\fi

A hypergraph $\calh=(\calv,\cale)$ is 
said to be {\it connected}
if for any two vertices $v_1,v_2$ in $\calh$, 
there exists a sequence of edges $e_0, e_1, \cdots,e_k$ 
in $\calh$ such that 
$v_1\in e_0$, $v_2\in e_k$ and 
$e_i\cap e_{i+1}\ne \emptyset$ holds for all 
$i=0,1,\cdots,k-1$.
Now assume that $\calh$ is connected. 
An edge $e$ in $\calh$ is 
called as a {\it bridge} of $\calh$
if $\calh-e$ (i.e., the hypergraph obtained from 
$\calh$ by removing $e$) 
is disconnected.
Let $B(\calh)$ be the set of bridges of $\calh$.
If $\calh$ is a graph (i.e., $|e|=2$ for all $e\in \cale$), then the spanning subgraph 
$(\calv, B(\calh))$ is connected if and only if 
$B(\calh)=\cale$ and $\calh$ is a tree.
However, if $\calh$ has some edge $e$ with $|e|\ge 3$,  
it is possible that $(\calv, B(\calh))$ is connected 
while $B(\calh)$ is a proper subset of $\cale$.
Such an example is given in Figure~\ref{f1} (a).

\begin{figure}[htbp]
  \centering
\includegraphics[scale=0.6]{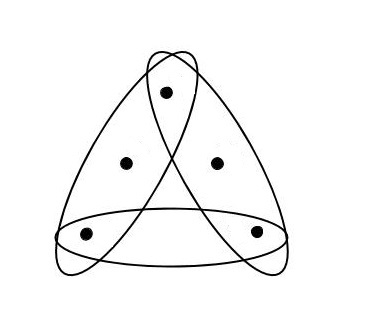}
\hspace{1 cm}
\includegraphics[scale=0.6]{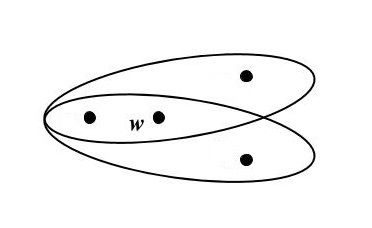}

(a) \hspace{5 cm} (b)
  \caption{Two hypergraphs}
\relabel{f1}
\end{figure}

\begin{T}\relabel{main3}
Let $\calh=(\calv,\cale)$ be any connected hypergraph.
If $B(\calh)$ is a proper subset of $\cale$
and the sub-hypergraph $(\calv, B(\calh))$ is connected,
then $\lambda^2$ is a factor of $P(\calh,\lambda)$.
\end{T}

It is also well known that for a connected graph $G$,
$G$ is separable 
(i.e., $G$ has a cut-vertex) if and only if 
$(\lambda-1)^2$ is a factor of $P(G,\lambda)$
(see \cite{whi1984,woo1977}).
But this property does not hold for chromatic polynomials
of hypergraphs (see Theorem~\ref{main4-0}).

We first need to make it clear what are  
separable hypergraphs. 
A connected hypergraph $\calh=(\calv,\cale)$ is said to be 
{\it separable} at a vertex $w$ 
if the hypergraph $\calh-w$ obtained from $\calh$ 
by removing $w$ and all edges containing $w$
is disconnected. 
This definition is a natural extension of the one for 
separable graphs. 
Observe that $\calh=(\calv,\cale)$ 
is separable at $w$
if and only if $\calv$ has two subsets 
$\calv_1$ and $\calv_2$
such that $\calv_1\cup \calv_2=\calv$, 
$\calv_1\cap \calv_2=\{w\}$
and for each $e\in \cale$, 
either $w\in e$ or $e\subseteq \calv_i$ 
for some $i\in \{1,2\}$.
Note that if $\calh$ is a graph, 
then $w\in e$ implies that $e\subseteq \calv_i$ for some $i$.
But if $\calh$ is not a graph, 
it is possible that $|e\cap \calv_i|\ge 2$ 
for both $i\in \{1,2\}$.

The hypergraph in Figure~\ref{f1} (b)
is connected and separable at $w$, but 
its chromatic polynomial is $\lambda(\lambda-1)
(\lambda^2+\lambda-1)$
which does not contain a factor $(\lambda-1)^2$.
Actually this hypergraph 
is contained in a family of connected and 
separable hypergraphs 
whose chromatic polynomials have no factor $(\lambda-1)^2$.
For any $\calh=(\calv,\cale)$,
let $F(\calh)$ be the set of vertices $w\in \calv$ 
such that $w\in e$ for every $e\in \cale$.
Observe that $w\in F(\calh)$ if and only if 
$e\not\subseteq \calv-\{w\}$ for every $e\in \cale$.
If $\calh$ does not have parallel edges (i.e., 
edges $e_1,e_2$ with $e_1=e_2$), 
then $F(\calh)=\calv$ if and only if 
$\calh$ is connected and $|\cale|=1$. 
It is trivial that if $|\cale|=1$, then 
$P(\calh,\lambda)$ does not have a factor $(\lambda-1)^2$.
We will show that if $|\cale|\ge 2$
and $F(\calh)\ne \emptyset$,
then $P(\calh,\lambda)$ has a factor $(\lambda-1)^2$
if and only if $|F(\calh)|=1$.

For a hypergraph $\calh=(\calv,\cale)$ and 
any $\calv_0\subseteq \calv$, 
let $\calh\cdot \calv_0$ denote the hypergraph 
obtained from $\calh$ by identifying all vertices in $\calv_0$
as one vertex 
and
let $\calh[\calv_0]$ be the hypergraph 
with vertex set $\calv_0$ and edge set 
$\{e\in \cale: e\subseteq \calv_0\}$.
We call $\calh[\calv_0]$ the sub-hypergraph of $\calh$ 
induced by $\calv_0$.
Let $\calh-\calv_0$ be the induced sub-hypergraph 
$\calh[\calv-\calv_0]$.
A hypergraph is said to be empty if it contains no edges. 
Let $\cali(\calh)$ be the set of those subsets $\calv_0$ 
of $\calv$ such that $\calh[\calv_0]$ is an empty graph.
A hypergraph $\calh$ is said to be {\it Sperner}
if $e_1\not\subseteq e_2$ for each pair of edges $e_1,e_2$ 
in $\calh$.

For the case that $F(\calh)=\emptyset$ and 
$\calh$ is separable,
we also give an equivalent statement 
for $P(\calh,\lambda)$ to have a factor $(\lambda-1)^2$.

\begin{T}\relabel{main4-0}
Let $\calh=(\calv,\cale)$ be any connected and Sperner 
hypergraph with $|\cale|\ge 2$.
\begin{enumerate}
\item[(i)] If $F(\calh)\ne \emptyset$,
then $P(\calh,\lambda)$ has a factor $(\lambda-1)^2$
if and only if $|F(\calh)|=1$;
\item[(ii)] 
If $F(\calh)=\emptyset$
and $\calh$ has 
a vertex $w$ and two proper subsets 
$\calv_1$ and $\calv_2$ of $\calv$ 
such that $\calv_1\cup \calv_2=\calv$, 
$\calv_1\cap \calv_2=\{w\}$
and for each $e\in \cale$, 
either $w\in e$ or $e\subseteq \calv_i$ for some $i$,
then $P(\calh,\lambda)$ has a factor $(\lambda-1)^2$
if and only if one of the following conditions is satisfied:
\begin{enumerate}
%\item $\calv_i\in \cali(\calh)$ for both $i=1,2$;
\item $\calv_i\notin \cali(\calh)$ for both $i=1,2$;
\item for some $i\in \{1,2\}$,  
$\calv_i\in \cali(\calh)$,
$\calv_{3-i}\notin \cali(\calh)$
and $P(\calh\cdot \calv_i,\lambda)$ has 
a factor $(\lambda-1)^2$.
\end{enumerate}
\end{enumerate}
\end{T}

We will prove Theorems~\ref{main1}-\ref{main4-0}
 in Sections 3-5
after some fundamental results are introduced 
in Section 2.
Finally, in Section 6, 
we will propose some open problems regarding 
multiplicities of 
roots ``$0$" and ``$1$" of $P(\calh,\lambda)$ for a hypergraph $\calh$ and some problems on the locations of real roots 
of chromatic polynomials of Zykov-planar hypergraphs,
where the definition of a Zykov-hypergraph was originally 
given by Zykov \cite{zyk2}, as stated in Definition~\ref{d1}.

\section{Preliminaries}

In this section, we present several known 
results on chromatic polynomials of hypergraphs, 
which will be applied later.
The first one follows directly from the definition of 
weak proper colourings of a hypergraph.

\begin{prop}\relabel{pro1}
Let $e_1,e_2$ be any two edges in a hypergraph $\calh$. 
If $e_1\subseteq e_2$, then 
$$
P(\calh,\lambda)=
P(\calh-e_2,\lambda),
$$
where $\calh-e_2$ is the hypergraph obtained from $\calh$ 
by removing $e_2$.
\end{prop}

By Proposition~\ref{pro1}, we need only to consider 
Sperner hypergraphs in the function $P(\calh,\lambda)$.

For a hypergraph $\calh=(\calv,\cale)$, a {\it component} of $\calh$ 
is an induced and connected sub-hypergraph $\calh[\calv_0]$
such that $\calh[\calv_0\cup \{v\}]$ is disconnected 
for any $v\in \calv-\calv_0$.
By the definition of $P(\calh,\lambda)$,
$P(\calh,\lambda)$ has the following factorization 
%we have the following expression for $P(\calh,\lambda)$ 
when $\calh$ is disconnected.

\begin{prop}\relabel{pro2}
Assume that $\calh_1,\cdots,\calh_k$ 
are components of $\calh$. Then 
$$
P(\calh,\lambda)
=\prod_{1\le i\le k}P(\calh_i,\lambda).
$$
\end{prop}

For any hypergraph $\calh=(\calv,\cale)$
and $\calv_0\subset \calv$, 
recall that $\calh\cdot \calv_0$ is obtained from $\calh$ 
by identifying all vertices in $\calv_0$ as one, i.e., 
$\calh\cdot \calv_0$ is the hypergraph 
%let $\calh\cdot \calv_0$ be the hypergraph 
with vertex set $(\calv-\calv_0)\cup \{w\}$
and edge set 
$$\{e'\in \cale: e'\cap \calv_0=\emptyset\}
\cup \{(e'-\calv_0)\cup \{w\}: 
e'\cap \calv_0\ne \emptyset\},
$$ 
where $w\notin \calv$.
For an edge $e$ in $\calh$,
let $\calh/e$ be the hypergraph $(\calh-e)\cdot e$.
This hypergraph $\calh/e$ is said to be 
obtained from $\calh$ by {\it contracting} $e$.

The deletion-contraction formula 
for chromatic polynomials of graphs
is very important for the computation of this polynomial
\cite{big1993, birk1946, D, rea1, rea2, whi1932b}. 
It was extended to chromatic polynomials of hypergraphs
by Jones~\cite{Jones 1976}.

\begin{T}[\cite{Jones 1976}] \relabel{t1}
%(Deletion-Contraction Theorem) 
Let $\calh=(\calv,\cale)$ be 
a hypergraph.  For any $e\in \cale$, 
\begin{equation}\relabel{e1}
P(\calh,\lambda)=P(\calh-e,\lambda)-P(\calh/e,\lambda).
\end{equation}
\end{T}

Note that Theorem~\ref{t1} can be equivalently stated below:
for any subset $e$ of $\calv$,  
\begin{equation}\relabel{e1-0}
P(\calh,\lambda)=
P(\calh+e,\lambda)+P(\calh\cdot e,\lambda),
\end{equation}
where $\calh+e$ is the hypergraph obtained from $H$ 
by adding a new edge $e$.

A hypergraph $\calh=(\calv,\cale)$ is written as 
$\calh_{1}\cup\calh_{2}$,
where $\calh_i=(\calv_i,\cale_i)$ is a hypergraph 
for $i=1,2$,
if $\calv=\calv_1\cup \calv_2$,  
$\cale=\cale_1\cup \cale_2$
and for any $e\subseteq \calv_1\cap \calv_2$,
$e\in \cale_1$ if and only if $e\in \cale_2$.
%$e\in \cale$, if $e\subseteq \calv_1\cap \calv_2$, then $e\in \cale_1\cap \cale_2$.
If $\{u,v\}\in \cale_1\cap \cale_2$ 
for each pair $\{u,v\}\subseteq \calv_1\cap \calv_2$,
then write $\calh_1\cap \calh_2=K_p$, 
where $p=|\cale_1\cap \cale_2|$.
Borowiecki and \L{}azuka~\cite{Borowiecki 2000}
extended Zykov's result \cite{zyk} 
on the factorization of $P(G_1\cup G_2,\lambda)$ 
when $G_1\cap G_2\cong K_p$ for two graphs $G_1$ and $G_2$.

\begin{T}[\cite{Borowiecki 2000}] \relabel{t2} 
If %$\calh$ is a hypergraph 
$\calh=\calh_{1}\cup\calh_{2}$ and $\calh_{1}\cap \calh_{2}=K_{p}$, then 
\begin{equation}
P(\calh,\lambda)=\frac{P(\calh_{1},\lambda)P(\calh_{2},\lambda)}{P(K_{p},\lambda)}. \label{e2}
\end{equation}
\end{T}

\section{Proof of Theorem~\ref{main1}
and Corollary~\ref{mainco1} (b)}

Let $G=(V,E)$ be a simple graph, 
i.e., $G$ has neither parallel edges nor loops. 
For $v\in V$, let $N_G(v)$ (or simply $N(v)$)
be the set $\{u\in V: uv\in E\}$,
and let $N[v]=N(v)\cup \{v\}$.
The degree of $v$ in $G$, denoted by $d(v)$, is
the size $|N(v)|$ of $N(v)$.

\begin{prop}\relabel{lem1}
For any vertex $v\in V$, 
%Let $\calh^{G}=(\calv,\cale)$ be a hypergraph and $v\in \calv$, then 
\begin{equation}\relabel{e13}
P(\calhd G,\lambda)=(\lambda-1)\cdot P(\calhd {G-\{v\}},\lambda)+(\lambda-1)^{d(v)}
\cdot P(\calhd {G-N[v]},\lambda).
\end{equation}
\end{prop}

\begin{proof}
Let $w$ be the new vertex in $\calhd G$ when it is produced 
from $G$.
%Let $\calhd G+e$ denote the hypergraph obtained from $\calhd G$ by adding a new edge $e$, where $e=\{w, v\}$. 
%and $\calhd G/e$ be the hypergraph obtained from $\calhd G$ by identifying $w$ and $v$.
By (\ref{e1-0}), %we have 
\begin{equation}\relabel{e14}
P(\calhd G,\lambda)=P(\calhd G+e,\lambda)
+P(\calhd G\cdot e,\lambda),
\end{equation}
where $e=\{w, v\}$. 

Observe that $e\subset \{w,v,v_i\}$ for all $v_i\in N(v)$.
By Proposition~\ref{pro1}, 
$$
P(\calhd G+e,\lambda)
=P(\calhd G+e-\cale(v),\lambda),
$$
where $\cale(v)=\{\{w,v,v_i\}: v_i\in N(v)\}$.
By Theorem~\ref{t2}, %(\ref{e3}), 
\begin{equation}\label{e15}
P(\calhd G+e-\cale(v),\lambda)
=(\lambda-1)\cdot P(\calhd G-\{v\},\lambda)
=(\lambda-1)\cdot P(\calhd {G-\{v\}},\lambda).
\end{equation}
Note that the edges 
$\{w,v,v_{i}\}$ in $\calhd G$, where $v_i\in N(v)$,  
are changed to $\{w, v_{i}\}$ in $\calhd G\cdot e$, 
and thus all edges $\{w, v_{i}, u\}$
in $\calhd G$, where $u\in N(v_i)-\{v\}$, 
can be removed by Proposition~\ref{pro1}. 
By Theorem \ref{t2} again, %(\ref{e3}) again,
\begin{equation}\relabel{e16}
P(\calhd G\cdot e,\lambda)=(\lambda-1)^{d(v)}
\cdot P(\calhd {G-N[v]},\lambda).
\end{equation}
Hence the result follows from (\ref{e14}), 
(\ref{e15}) and (\ref{e16}).
\end{proof}

The following property on the independence polynomial of a graph
is needed for proving Theorem~\ref{main1}.

\begin{prop}[\cite{M}]\relabel{p3.1}
Let $G=(V,E)$ be any simple graph and $v\in V$. Then 
$I(G,x)=I(G-\{v\},x)+xI(G-N[v],x)$.
\end{prop}

Now we are ready to prove Theorem~\ref{main1}
by applying Propositions~\ref{lem1} and~\ref{p3.1}.

\vspace{0.5 cm}

\noindent {\it Proof of Theorem~\ref{main1}}: 
Suppose that the result fails.
Assume that $G=(V,E)$ is a simple graph 
for which the result fails and 
$|V|+|E|$ has the minimum value
among all those graphs for which the result fails. 
We shall complete the proof by showing the following claims.

\noindent {\bf Claim 1}: $n=|V|\ge 2$.

Assume that $n=|V|=1$. Then $E=\emptyset$ as 
$G$ is simple. 
Thus $\calhd G$ is the hypergraph with two vertices 
and no edges, implying that 
$P(\calhd G,\lambda)=\lambda^2$.
Observe that the right-hand side 
of (\ref{main1-eq1}) is 
$$
\lambda(\lambda-1) (1+1/(\lambda-1))=\lambda^2,
$$
implying that Theorem~\ref{main1} holds for this graph, 
a contradiction.

\noindent {\bf Claim 2}: $G$ does not exist.

Let $v$ be any vertex of $G$, by Proposition~\ref{lem1}, we have 
\begin{equation}\relabel{e19}
P(\calhd G,\lambda)=(\lambda-1)\cdot 
P(\calhd {(G-\{v\})},\lambda)
+(\lambda-1)^{d(v)}\cdot 
P(\calhd {(G-N[v])},\lambda).
\end{equation}
By the assumption on $G$,
Theorem~\ref{main1} holds for both 
$G-\{v\}$ and $G-N[v]$. Thus 
\begin{equation}\relabel{e20}
P(\calhd {(G-\{v\})},\lambda)=\lambda \cdot (\lambda-1)^{n-1}\cdot I(G-v,\frac{1}{\lambda-1})
\end{equation}
and 
\begin{equation}\label{e21}
P(\calhd {(G-N[v])},\lambda)
=\lambda \cdot (\lambda-1)^{n-d(v)-1}\cdot I(G-N[v],\frac{1}{\lambda-1}).
\end{equation}
From Proposition \ref{p3.1}
and equalities 
(\ref{e19}), (\ref{e20}) and (\ref{e21}), we obtain 
\begin{align*}
P(\calhd G,\lambda)
%&=& (\lambda-1)\cdot F(G-\{v\},\lambda)+(\lambda-1)^{d(v)}\cdot F(G-N[v],\lambda)\\
=& \lambda(\lambda-1)^{n}\cdot I(G-\{v\},\frac{1}{\lambda-1})\\
&+ \lambda\cdot (\lambda-1)^{n-1}\cdot I(G-N[v],\frac{1}{\lambda-1})\\
=&\lambda \cdot (\lambda-1)^{n} \cdot I(G,\frac{1}{\lambda-1}).
\end{align*}
Thus equality (\ref{main1-eq1}) holds for $G$, 
a contradiction. 
Therefore Claim 2 is proved
and Theorem~\ref{main1} holds.
\hfill {$\Box$}

We end this section by providing a proof
of Corollary~\ref{mainco1} (b).
\vspace{0.3 cm}
\\
\noindent{\it Proof of Corollary~\ref{mainco1} (b)}.
It has been shown in \cite{B}
that the real roots of independence polynomials 
are dense in the interval $(-\infty, 0]$. 
Then Theorem~\ref{main1} implies that 
the real roots of chromatic polynomials of hypergeraphs 
$\calho$ for all graphs $G$ 
are dense in the interval 
$(-\infty, 1]$. 
By Corollary~\ref{pro3-co1}, 
which follows from Proposition~\ref{pro3} in Section 5 directly, 
the real roots of 
the chromatic polynomials of hypergraphs $\calho+K_1$ for all 
graphs $G$ are dense in the interval $(-\infty, 2]$,
where $\calho+K_1$  is the hypergraph obtained from 
$\calho$ by adding a new vertex $u$ 
and adding new edges $\{u,v\}$ for 
all vertices $v$ in $\calho$.
Repeating this process or applying the fact that 
the real roots of chromatic polynomials of graphs
are dense in $[2,\infty)$, 
the result of Corollary~\ref{mainco1} (b) holds.
\hfill {$\Box$}

\section{Proof of Theorem~\ref{main2}}

Let $G=(V,E)$ be a graph of order $n$ and size $m$.
Assume that $G$ may have loops or parallel edges.
We first establish the following recursive 
formula for $P(\calh_G,\lambda)$.

\begin{prop}\relabel{l1} 
For any $e\in E(G)$,
\begin{equation}\relabel{e4}
P(\calh_{G},\lambda)=\lambda P(\calh_{G-e},\lambda)-P(\calh_{G/e},\lambda).
\end{equation}
\end{prop}

\begin{proof}
Assume that $u$ and $v$ are the two ends of $e$ in $G$.
It is possible that $u=v$, as $e$ may be a loop.
Let $e'=\{u,v,e\}$.
So $e'$ is the edge in $\calh_G$ corresponding to $e$.
By Theorem~\ref{t1}, we have 
\begin{equation}\relabel{eq20}
P(\calh_{G},\lambda)=P(\calh_{G}-e',\lambda)-P(\calh_{G}/e',\lambda)
\end{equation}
Note that $\calh_{G}-e'$ consists of an isolated vertex and the hypergraph $\calh_{G-e}$,
and $\calh_{G}/e'$ is actually the hypergraph 
$\calh_{G/e}$.
Thus the result holds. 
\end{proof}

Listed below are some other properties of $P(\calh_{G},\lambda)$ which 
can be proved easily by applying 
Theorem \ref{t2} and Proposition \ref{l1}.

\begin{prop}\relabel{p1}
Let $G$ be a multigraph of order $n$.
\begin{enumerate}
\item[(a)] If $G$ is an empty graph, then $P(\calh_{G}, \lambda)=\lambda^{n}$;

\item [(b)] If $e$ is a loop of $G$, then $P(\calh_{G}, \lambda)= (\lambda-1) P(\calh_{G-e}, \lambda)$;

\item [(c)] If $e$ is a bridge of $G$,, then $P(\calh_{G}, \lambda)= (\lambda^{2}-1) P(\calh_{G/e}, \lambda)$.
\end{enumerate}
\end{prop}

Some fundamental properties on Tutte polynomials $T_G(x,y)$
are needed for proving Theorem~\ref{main2}.

%Let $T(G;x,y)$ be the Tutte polynomial of the graph $G$. The following propositions will be found in \cite{WD} and \cite{J}.

\begin{prop}[\cite{J, WD}]\relabel{p2}
Let $G$ be a multigraph.  
\begin{enumerate}
\item[(a)] If $G$ is an empty graph, then $T_G(x,y)=1$; 

\item [(b)] If $e$ is a loop of $G$, then $T_G(x,y)=y\cdot T_{G-e}(x,y)$;

\item [(c)] If $e$ is a bridge of $G$, 
then $T_G(x,y)=x\cdot T_{G/e}(x,y)$;

\item [(d)] For any $e\in E(G)$, if $e$ is neither a loop nor a bridge, then 
$T_G(x,y)=T_{G-e}(x,y)+T_{G/e}(x,y)$.
%\item [(e)] If $G_1, G_2,\cdots, G_r$ are components of $G$,then $T_G(x,y)=\prod_{1\le i\le r}T_{G_i}(x,y)$. 
\end{enumerate}
\end{prop}

We are now ready to prove Theorem~\ref{main2}.

\vspace{0.5 cm}

\noindent {\it Proof of Theorem~\ref{main2}}:
%Let $f(G,\lambda)=P(\calh_{G},\lambda)$. 
We will prove the result by induction on the size $m$ of $G$. 

If $m=0$, Theorem~\ref{main2} holds for $G$ 
by Propositions \ref{p1} (a) and \ref{p2} (a).

Assume that Theorem~\ref{main2} holds for any graph 
of size less than $m$, where $m>0$.
Now we assume that $G=(V,E)$ is a graph of size $m$.
Let $e$ be any edge in $G$.

\noindent {\bf Case 1}: $e$ is a loop.

By Propositions~\ref{p1} (b) and~\ref{p2} (b),
\begin{equation}\label{e50}
P(\calh_G,\lambda)=(\lambda-1)P(\calh_{G-e},\lambda),
\quad
T_G(x,y)=yT_{G-e}(x,y).
\end{equation}
By the inductive assumption, 
Theorem~\ref{main2} holds for $G-e$, i.e., 
\begin{equation}\label{e5}
P(\calh_{G-e},\lambda)=\lambda^{m-1-n+2c(G)}
\cdot (-1)^{n+c(G)}\cdot 
T_{G-e}(1-\lambda^{2},(\lambda-1)/{\lambda}).
\end{equation}
Thus Theorem~\ref{main2} holds for $G$
by equalities in (\ref{e50}) and (\ref{e5}).

\noindent {\bf Case 2}: $e$ is a bridge.

By Propositions~\ref{p1} (c) and~\ref{p2} (c),
\begin{equation}\label{e510}
P(\calh_G,\lambda)=(\lambda^2-1)P(\calh_{G/e},\lambda),
\quad
T_G(x,y)=xT_{G/e}(x,y).
\end{equation}
By the inductive assumption, 
Theorem~\ref{main2} holds for $G/e$, i.e., 
\begin{equation}\label{e51}
P(\calh_{G/e},\lambda)=\lambda^{m-n+2c(G)}
\cdot (-1)^{n-1+c(G)}\cdot 
T_{G/e}(1-\lambda^{2},(\lambda-1)/{\lambda}).
\end{equation}
Thus Theorem~\ref{main2} holds for $G$
by equalities in (\ref{e510}) and (\ref{e51}).

\noindent {\bf Case 3}: $e$ is neither a bridge nor a loop.

Then $c(G-e)=c(G/e)=c(G)$.
By inductive assumption, 
Theorem~\ref{main2} holds for $G-e$ and $G/e$, i.e., 
\begin{equation}\relabel{e9}
P(\calh_{G-e},\lambda)=\lambda^{m-1-n+2c(G)}\cdot (-1)^{n+c(G)}\cdot T_{G-e}(1-\lambda^{2},({\lambda-1})/{\lambda})
\end{equation}
and 
\begin{equation}\relabel{e10}
P(\calh_{G/e},\lambda)
=\lambda^{m-n+2c(G)}\cdot (-1)^{n-1+c(G)}\cdot 
T_{G/e}(1-\lambda^{2},({\lambda-1})/{\lambda}).
\end{equation}
By (\ref{e9}), (\ref{e10}) and Proposition~\ref{p2} (d),
it can be verified that Theorem~\ref{main2} holds for $G$.
\hfill $\Box$

\section{Proofs of Theorems~\ref{main3} and~\ref{main4-0} }

In this section, we will complete the proofs of 
Theorems~\ref{main3} and~\ref{main4-0}.

\vspace{0.3 cm}

\noindent {\it Proof of Theorem~\ref{main3}}: 
%Let $\calh'$ denote the spanning sub-hypergraph $(\calv,B(\calh))$.
By a result due to Tomescu \cite{Tomescu 1998},
the coefficient of $\lambda$ in 
$P(\calh,\lambda)$ is equal to 
$$
a_1=\sum_{j} (-1)^j N_{j}
$$
where $N_{j}$ is the number of 
connected and spanning sub-hypergraphs of $\calh$
with exactly $j$ edges.
By the given conditions, 
any sub-hypergraph $(\calv, \cale')$ of $\calh$
%induced by a set $\cale'$ of edges in $\calh$ 
is connected if and only if 
$B(\calh)\subseteq \cale'$.
Assume that 
$r=|B(\calh)|$ and $k=|\cale|-|B(\calh)|>0$. Then 
$N_j={k\choose j-r}$ and 
$$
a_1=\sum_{j=r}^{r+k} (-1)^{j} {k\choose j-r}
=0.
$$  
Thus the result holds.
\hfill $\Box$
%\end{proof}

Some results are needed for  proving 
Theorem~\ref{main4-0}.
Let $\Phi(\calh)$ be the set of those partitions 
$\{\calv_1,\cdots,\calv_k\}$ of $\calv$ 
such that each $\calv_i$ is an non-empty 
member in $\cali(\calh)$.
%where  $\cali(\calh)=\{\calv_0\subseteq \calv: \forall e\in \cale, e\not\subseteq \calv_0\}$.
By the definition of $P(\calh,\lambda)$,  
\begin{equation}\relabel{e11}
P(\calh,\lambda)
=%\sum_{k\ge 1}
\sum_{\{\calv_1,\cdots,\calv_k\} \in \Phi(\calh)}
(\lambda)_k,
\end{equation}
where $(x)_0=1$ and 
$(x)_k=x(x-1)\cdots (x-k+1)$ for any number $x$ 
and positive integer $k$.
By (\ref{e11}), we can deduce the following result.

\begin{prop}\relabel{pro3}
Let $w$ be a fixed vertex in $\calh=(\calv,\cale)$.
Then 
\begin{equation}\relabel{e7}
P(\calh,\lambda)
=\lambda\sum_{w\in \calv_0\in {\cal I}(\calh)}
P(\calh-\calv_0,\lambda-1).
\end{equation}
\end{prop}

\begin{proof}
By (\ref{e11}), we can assume that 
\begin{equation}\relabel{e12}
P(\calh,\lambda)
=\sum_{\{\calv_0,\calv_1,\cdots,\calv_k\} \in \Phi(\calh)
\atop w\in \calv_0}
(\lambda)_{k+1}.
\end{equation}
%where $w\in \calv_0$ and $k\ge 0$.
%Let $\calh'=\calh-\calv_0$. 
For any $\calv_0\in \cali(\calh)$ with $w\in \calv_0$,
$\{\calv_0, \calv_1,\cdots,\calv_k\}\in \Phi(\calh)$ 
if and only if $\{\calv_1,\cdots,\calv_k\}
\in \Phi(\calh-\calv_0)$.
Thus 
\begin{equation}\relabel{e40}
P(\calh,\lambda)
=\lambda 
\sum_{w\in \calv_0\in \cali(\calh)}
\sum_{\{\calv_1,\cdots,\calv_k\} \in \Phi(\calh-\calv_0)}
(\lambda-1)_{k}.
\end{equation}
By (\ref{e11}) again, the result holds.
\end{proof}

By Proposition~\ref{pro3}, $\lambda$ is a factor of 
$P(\calh,\lambda)$ for any hypergraph $\calh$.
Actually $\lambda-1$ is also a factor of 
$P(\calh,\lambda)$ whenever $\calh$ is not an empty graph.

\begin{C}\relabel{pro3-co2}
For any hypergraph $\calh=(\calv,\cale)$,
$\lambda^{c(\calh)}$ and $(\lambda-1)^{c'(\calh)}$
are factors of $P(\calh,\lambda)$,
where $c(\calh)$ is the number of components of $\calh$
and $c'(\calh)$ is the number of those components of $\calh$
which contain edges.
\end{C}

\begin{proof}
We just prove that 
$(\lambda-1)^{c'(\calh)}$ is a factor of $P(\calh,\lambda)$.
It suffices to show that if $\calh$ is connected and 
non-empty, then 
$\lambda-1$ is a factor of $P(\calh,\lambda)$.
As $\calh$ is not empty, $\calv\notin \cali(\calh)$.
Thus $P(\calh-\calv_0,\lambda)$ has a factor $\lambda$ 
for every $\calv_0\in \cali(\calh)$.
By Proposition~\ref{pro3}, 
$\lambda-1$ is a factor of $P(\calh,\lambda)$.
\end{proof}

%\vspace{0.3 cm}

Recall that for a hypergraph ${\cal H}$,
 ${\cal H}+K_1$ is the hypergraph obtained from  ${\cal H}$
 by adding a new vertex $u$ and adding new edges 
 $\{u,v\}$ for all vertices $v$ in  ${\cal H}$.
 By Proposition~\ref{pro3}, the following result is obtained.

\begin{C}\relabel{pro3-co1}
For any hypergraph ${\cal H}$, 
$P({\cal H}+K_1,\lambda)=\lambda P({\cal H},\lambda-1).$
\end{C}

By Corollary~\ref{pro3-co1},
$P({\cal H}+K_1,\lambda)$ has a factor $(\lambda-1)^2$ 
if and only if $P({\cal H},\lambda)$ has a factor $\lambda^2$.
If $\calh$ is connected, then 
$\calh+K_1$ is not separable. 
By Theorem~\ref{main3}, %and Corollary~\ref{pro3-co1},
there exist non-separable hypergraphs 
whose chromatic polynomials have a factor $(\lambda-1)^2$.

Now we establish two important results for 
proving Theorem~\ref{main4-0}.

\begin{prop}\relabel{main4}
Let $\calh=(\calv,\cale)$ be any connected hypergraph
with a vertex $w$ and two proper subsets 
$\calv_1$ and $\calv_2$ of $\calv$ 
such that $\calv_1\cup \calv_2=\calv$, 
$\calv_1\cap \calv_2=\{w\}$
and for each $e\in \cale$, 
either $w\in e$ or $e\subseteq \calv_i$ for some $i$.
Then $(\lambda-1)^2$ is a factor of $P(\calh,\lambda)$
if and only if $\lambda^2$ is a factor of 
the following polynomial:
\begin{equation}
\sum_{\calv_0\in \cali_{(\calv_1,\calv_2)}(\calh)}
P(\calh-\calv_0,\lambda),
\end{equation}
where $\cali_{(\calv_1,\calv_2)}(\calh)$ is the set of those 
$\calv_0\in \cali(\calh)$ with $\calv_i\subseteq \calv_0$ 
for some $i\in \{1,2\}$.
\end{prop}

\begin{proof}
By Proposition~\ref{pro3}, 
\begin{equation}\relabel{e41}
P(\calh,\lambda)
=\lambda \sum_{\calv_0\in \cali_{(\calv_1,\calv_2)}}
P(\calh-\calv_0,\lambda-1)
+\lambda \sum_{w\in \calv_0\in \cali(\calh)\atop 
\calv_0\notin \cali_{(\calv_1,\calv_2)}}
P(\calh-\calv_0,\lambda-1).
\end{equation}
For any $\calv_0\in \cali(\calh)$ 
with $w\in \calv_0$,
if $\calv_0\notin \cali_{(\calv_1,\calv_2)}(\calh)$, 
then $\calh-\calv_0$ is disconnected and 
Corollary~\ref{pro3-co2} implies that 
$\lambda^2$ is a factor of $P(\calh-\calv_0,\lambda)$
and thus
$(\lambda-1)^2$ is a factor of $P(\calh-\calv_0,\lambda-1)$.
Hence the result follows from (\ref{e41}).
\end{proof}
%\hfill $\Box$

\vspace{0.3 cm}

Let $\cali'_{(\calv_1,\calv_2)}(\calh)$ 
be the set of those members $\calv_0$ in 
$\cali_{(\calv_1,\calv_2)}(\calh)$
such that $\calh-\calv_0$ is connected. 
For each 
$\calv_0\in \cali_{(\calv_1,\calv_2)}(\calh)-
\cali'_{(\calv_1,\calv_2)}(\calh)$,
$\calh-\calv_0$ is disconnected and thus 
$P(\calh-\calv_0,\lambda)$ has a factor $\lambda^2$ 
by Corollary~\ref{pro3-co2}. 
By Proposition~\ref{main4}, we get the following result.

\begin{C}\relabel{main4-cor3}
Let $\calh=(\calv,\cale)$ be any connected hypergraph
with  a vertex $w$ and two proper subsets 
$\calv_1$ and $\calv_2$ of $\calv$ 
such that $\calv_1\cup \calv_2=\calv$, 
$\calv_1\cap \calv_2=\{w\}$
and for each $e\in \cale$, 
either $w\in e$ or $e\subseteq \calv_i$ for some $i$.
Then $(\lambda-1)^2$ is a factor of $P(\calh,\lambda)$
if and only if $\lambda^2$ is a factor of 
\begin{equation}\relabel{main4-cor3-eq1}
\sum_{\calv_0\in \cali'_{(\calv_1,\calv_2)}(\calh)}P(\calh-\calv_0,\lambda).
\end{equation}
\end{C}

We are now going to prove Theorem~\ref{main4-0}.

\noindent {\it Proof of Theorem~\ref{main4-0}}: 
(i) First assume that $F(\calh)=\{w\}$.
As $|\cale|\ge 2$ and $\calh$ is Sperner, $|\calv|\ge 3$.
Then $\calh$ is separable at $w$.
Assume that $\calv_1$ and $\calv_2$ are 
proper subsets  of $\calv$ such that 
$\calv_1\cap \calv_2=\{w\}$ and 
$\calv_1\cup \calv_2=\calv$.
As $w$ is the only member in $F(\calh)$, 
$\calv-\{u\}\notin \cali(\calh)$ for every $u\in \calv-\{w\}$.
Thus, for each $\calv_0\in \cali(\calh)$ with $w\in \calv_0$,
we have $|\calv_0|\le |\calv|-2$ and so
$\calh-\calv_0$ is an empty graph of order at least $2$,
implying that $\lambda^2$ is a factor of 
$P(\calh-\calv_0,\lambda)$.
By Proposition~\ref{pro3} or Proposition~\ref{main4},
$(\lambda-1)^2$ is a factor of $P(\calh,\lambda)$.

Now consider the case that $k=|F(\calh)|\ge 2$.
%For any $w_1,w_2\in F(\calh)$, 
By (\ref{e1-0}), we have 
\begin{equation}\relabel{main4-0-eq0}
P(\calh,\lambda)=P(\calh+F(\calh),\lambda)
+P(\calh\cdot F(\calh),\lambda).
\end{equation}
As $F(\calh)\subseteq e$ for each 
$e\in \cale$, Proposition~\ref{pro1} implies that 
\begin{equation}\relabel{main4-0-eq00}
P(\calh+F(\calh),\lambda)=
P(\calh_0,\lambda)=
\lambda^{|\calv|-k}(\lambda^k-\lambda),
\end{equation}
where $\calh_0$ is the hypergraph with vertex set $\calv$ 
and edge set $\{F(\calh)\}$.
By (\ref{main4-0-eq0}) and (\ref{main4-0-eq00}),
\begin{equation}\relabel{main4-0-eq1}
P(\calh,\lambda)=
\lambda^{|\calv|-k}(\lambda^k-\lambda)
+P(\calh\cdot F(\calh),\lambda).
\end{equation}
Observe that $\calh\cdot F(\calh)$ is Sperner, connected 
and has as many edges as $\calh$. 
As $\calh\cdot F(\calh)$ has at least two edges and 
$|F(\calh\cdot F(\calh))|=1$,
$P(\calh\cdot F(\calh),\lambda)$ has a factor 
$(\lambda-1)^2$ by the result proved above.
Since $\lambda^{|\calv|-k}(\lambda^k-\lambda)$
does not have a  factor $(\lambda-1)^2$,
(\ref{main4-0-eq1}) implies that 
$P(\calh,\lambda)$ does not have a factor 
$(\lambda-1)^2$. 

(ii) As $F(\calh)=\emptyset$, it is impossible 
that $\calv_i\in \cali(\calh)$ for both $i=1,2$.
By Proposition~\ref{main4},
if $\calv_i\notin \cali(\calh)$ for both $i=1,2$, 
then $(\lambda-1)^2$ is a factor of $P(\calh,\lambda)$.

Now we assume that $\calv_1\in \cali(\calh)$ 
but $\calv_2\notin \cali(\calh)$.
By Proposition~\ref{main4}, 
$(\lambda-1)^2$ is a factor of $P(\calh,\lambda)$
if and only if $\lambda^2$ is a factor of the following 
polynomial:
\begin{equation}\relabel{main4-0-eq2}
\sum_{\calv_0\in \cali_{\calv_1}(\calh)}
P(\calh-\calv_0,\lambda),
\end{equation}
where 
$\cali_{\calv_1}(\calh)$ is the set of 
those $\calv_0\in \cali(\calh)$ 
with $\calv_1\subseteq \calv_0$.
Observe that 
\begin{equation}\relabel{main4-0-eq3}
\sum_{\calv_0\in \cali_{\calv_1}(\calh)}
P(\calh-\calv_0,\lambda)
=\sum_{\calv_1\cup \calv'\in \cali(\calh)
\atop \calv'\subseteq \calv_2-\{w\}}
P(\calh-(\calv_1\cup \calv'),\lambda).
\end{equation}
Let $\calh_0$ denote the hypergraph $\calh\cdot \calv_1$
and let $w_0$ denote the vertex in $\calh_0$ 
which is produced after identifying 
all vertices $\calv_1$ as one. 
Thus the vertex set of $\calh_0$ is 
$(\calv_2-\{w\})\cup \{w_0\}$.
Observe that for any $\calv'\subseteq \calv_2-\{w\}$,
$\calv_1\cup \calv'\in \cali(\calh)$ 
if and only if $\{w_0\}\cup \calv'\in \cali(\calh_0)$,
and  $\calh-(\calv_1\cup \calv')$ is 
exactly the hypergraph $\calh_0-(\{w_0\}\cup \calv')$.
Thus, by (\ref{main4-0-eq3}),
\begin{equation}\relabel{main4-0-eq4}
\sum_{\calv_0\in \cali_{\calv_1}(\calh)}
P(\calh-\calv_0,\lambda)
=\sum_{\calv'\cup\{w_0\}\in \cali(\calh_0)}
P(\calh_0-(\calv'\cup\{w_0\}),\lambda).
\end{equation}
By Proposition~\ref{pro3}, the right-hand side 
of (\ref{main4-0-eq4}) has a factor of $\lambda^2$ 
if and only if $P(\calh_0,\lambda)$ has a factor 
$(\lambda-1)^2$.

Hence (ii) holds.
\hfill $\Box$

For any graph $G$, 
if $G$ has two edges $e_1,e_2$
which have no any common end, 
then $\calhd{G}$ is separable at vertex $w$ which 
is the vertex not in $G$ and 
$\cali_{(\calv_1,\calv_2)}(\calhd{G})=\emptyset$
for suitable $\calv_1,\calv_2$ with $e_i\subseteq \calv_i$
for $i=1,2$,
implying that 
$P(\calhd{G},\lambda)$ has a factor 
$(\lambda-1)^2$ 
by Proposition~\ref{main4}.

By Theorem~\ref{main4-0}, we can easily get 
examples of separable hypergraphs $\calh$ 
whose chromatic polynomials don't have a factor $(\lambda-1)^2$.
Let $\calh$ be a hypergraph with vertex set 
$\{w\}\cup \{x_i: 1\le i\le s\}\cup 
\{y_j: 1\le j\le t\}$ and edge set 
\begin{equation}\relabel{main4-0-eq5}
\{\{y_j: 1\le j\le t\}\} 
\cup
\{\{w, x_1,x_2,\cdots,x_s,y_j\}: 1\le j\le t\},
\end{equation}
where $s\ge 1$ and $t\ge 1$.
Then $\calv_1=\{w\}\cup \{x_i: 1\le i\le s\}$ is a member 
of $\cali(\calh)$ while 
$\calv_2=\{w\}\cup \{y_j: 1\le j\le t\}$ is not.
By Theorem~\ref{main4-0}, $P(\calh,\lambda)$ has a factor 
$(\lambda-1)^2$ if and only if 
$P(\calh\cdot \calv_1,\lambda)$ has a factor 
$(\lambda-1)^2$.
Observe that $\calh\cdot \calv_1$ is a 
hypergraph with vertex set 
$\{w\}\cup \{y_j: 1\le j\le t\}$ and edge set 
\begin{equation}\relabel{main4-0-eq6}
\{\{y_j: 1\le j\le t\}\} \cup
\{\{w,y_j\}: 1\le j\le t\}.
\end{equation}
By Corollary~\ref{pro3-co1}, 
\begin{equation}\relabel{main4-0-eq7}
P(\calh\cdot \calv_1,\lambda)
=\lambda P(\calh-\calv_1,\lambda-1)
=\lambda ((\lambda-1)^t-(\lambda-1)),
\end{equation}
which does not have a factor $(\lambda-1)^2$.

Theorem~\ref{main4-0} actually provides 
a method of producing a separable hypergraph $\calh$ 
from a given hypergraph $\calh'$ such that 
$P(\calh,\lambda)$ has a factor $(\lambda-1)^2$
if and only if $P(\calh',\lambda)$ has a factor $(\lambda-1)^2$.
For any connected and Sperner 
hypergraph $\calh'$ with at least two edges
and any vertex $w$ in $\calh'$,
let $\calh$ be a connected hypergraph 
with vertex set 
$\calv(\calh')\cup S$, where $S$ is a non-empty set,
and edge set 
\begin{equation}\relabel{main4-0-eq8}
(\cale(\calh')-\cale')\cup 
\{e\cup S_e: e\in \cale', S_e\subseteq S\},
%\cup \{e\in \cale(\calh'):  e\notin \cale'\},
\end{equation}
where $\cale'$ is a set of some edges $e\in \cale(\calh')$ 
with $w\in e$ and $S_e$ is any subset of $S$.
As $\calh$ must be connected, 
$\bigcup_{e\in \cale'} S_e=S$.
By Theorem~\ref{main4-0},  
$P(\calh,\lambda)$ has a factor $(\lambda-1)^2$
if and only if $P(\calh',\lambda)$ has a factor $(\lambda-1)^2$.

\section{Conclusions and further research}

It is well known that for any simple graph $G$,
the multiplicity of root ``$0$'' of 
$P(G,\lambda)$ is equal to the number of components 
of $G$ (\cite{eis1972, rea1}).
However, Theorem~\ref{main3} shows that 
for a connected hypergraph $\calh$, 
$P(\calh,\lambda)$ may have a factor $\lambda^2$,
implying that  the multiplicity of root ``$0$'' of 
the chromatic polynomial of a hypergraph can be as large as 
twice the number of components of this hypergraph. 
%chromatic polynomials of hypergraphs have different propertyon the the multiplicity of root ``$0$''.
%this property does not hold for hypergraphs.
But we don't know what is the largest possible 
multiplicity of root ``$0$'' of $P(\calh, \lambda)$
%the chromatic polynomial of a hypergraph 
and the relation between 
the multiplicity of root ``$0$''  of 
$P(\calh, \lambda)$ and the structure of $\calh$.
Thus we propose the following problem 
regarding the multiplicity of root ``$0$'' of 
$P(\calh,\lambda)$ for a hypergraph $\calh$. 
%the following problems are worth for further study.

\begin{Pro} \relabel{prob2}
Let $\calh$ be a connected hypergraph. 
\begin{enumerate}
\item[(a)] What is a necessary and 
sufficient condition for $P(\calh,\lambda)$
to have a factor $\lambda^2$?

\item[(b)] Is it possible that 
the multiplicity of root ``$0$'' of 
$P(\calh,\lambda)$
is larger than $2$
for some non-separable hypergraph $\calh$
?

\item[(c)] 
What is the relation between the multiplicity of root ``$0$'' of 
$P(\calh,\lambda)$ 
and the structure of $\calh$?
\end{enumerate}
\end{Pro}

Note that for a bridge $e$ in a connected hypergraph 
$\calh$, by Theorem~\ref{t1}, 
$P(\calh,\lambda)$ has a factor $\lambda^2$ 
if and only if 
$P(\calh/e,\lambda)$ has a factor $\lambda^2$.
Thus the study of Problem \ref{prob2}(a) can be focused 
on those connected hypergraphs without 
bridges.

It is also well known that for a connected graph $G$, 
$(\lambda-1)^2$ is a factor of $P(G,\lambda)$ 
if and only if $G$ is separable (i.e., 
$G$ has a cut-vertex), and 
the multiplicity of root ``$1$'' of 
$P(G,\lambda)$ is equal to the number of blocks of $G$
 (\cite{whi1984, woo1977}).
However, such result does not hold for hypergraphs.
Theorem~\ref{main4-0} shows that 
$(\lambda-1)^2$ is a factor of $P(\calh,\lambda)$
for some but not all 
connected and separable hypergraphs $\calh$.
For connected but non-separable hypergraphs, 
%Corollary~\ref{main3-cor2} 
Theorem~\ref{main3} and Corollary~\ref{pro3-co1} also imply 
their chromatic polynomials 
may have a factor $(\lambda-1)^2$.
Regarding the multiplicity of root ``$1$'' of 
$P(\calh,\lambda)$ for a connected hypergraph $\calh$, 
we propose the following problem.

\begin{Pro}  \relabel{prob3}
Let $\calh$ be a connected hypergraph. 
\begin{enumerate}
\item[(a)] 
What is a necessary and 
sufficient condition for $P(\calh,\lambda)$
to have a factor $(\lambda-1)^2$?
\item[(b)]
What is the relation between the multiplicity of root ``$1$'' of 
$P(\calh,\lambda)$ 
and the structure of $\calh$?
\end{enumerate}
\end{Pro}

Properties (B.1) and (B.2) in Section 1 %page~\pageref{PN12} 
%for chromatic polynomials of graphs 
imply that $P(G,\lambda)$ does not have negative 
real roots for every graph $G$.
This property was only extended to 
a set of linear hypergraphs by Dohmen~\cite{Dohmen 1995}
who showed that 
$P(\calh,\lambda)$ have no negative real roots 
if each edge in a linear hypergraph $\calh$ has an even size 
and each cycle in $\calh$ has an edge of size $2$,
although Corollary~\ref{mainco1} tells that 
chromatic polynomials of hypergraphs have dense 
roots in the interval $(-\infty,0)$.
It is possible that Dohmen's result in \cite{Dohmen 1995} 
can be extended to a larger family of hypergraphs.
Now we propose the following problem
on negative real roots of chromatic polynomials
of hypergraphs.

\begin{Pro}  \relabel{prob4}
Let $\calh$ be a connected hypergraph of order $n$. 
\begin{enumerate}
\item[(a)] 
What is a necessary and 
sufficient condition for $P(\calh,\lambda)$
to have no negative roots?
\item[(b)] Is it true that if $\calh$ contains 
even-size edges only, then $P(\calh,\lambda)$
has no negative roots?
\item[(c)] Determine a function $f(n)$ 
such that $P(\calh,\lambda)$ has no roots 
in $(-\infty, f(n))$.
\end{enumerate}
\end{Pro}

The study of chromatic polynomials of planar graphs 
is very important in the topic of 
chromatic polynomials of graphs.
We will end this paper with some problems on 
the chromatic polynomials of planar hypergraphs. 
The planarity of hypergraphs has been studied by some researchers 
such as Heise, Panagiotou, Pikhurko and Taraz ~\cite{hei 2014}, 
Johnson and Pollak~\cite{Joh 1987},
Verroust and Viaud~\cite{ver 2004}, 
Zykov~\cite{zyk2}, etc.
Zykov~\cite{zyk2} and %some other definitions were given by 
Johnson and Pollak~\cite{Joh 1987}
gave different definitions for planar hypergraphs.
Here we take Zykov's planarity which seems more natural. 
His definition associates edges in a hypergraph with 
faces in a plane graph. 

\begin{D}[\cite{zyk2}] \relabel{d1}
Let $\calh=(\calv,\cale)$ be a hypergraph.
If there is a plane graph $G$ with vertex set $\calv$ 
such that for each $e\in \cale$, 
$e$ is the set of vertices in some face of $G$, 
then $\calh$ is said to be {\it Zykov-planar}. 
\end{D}

For example, if $\calh=(\calv,\cale)$ is a hypergraph
with $\calv=\{a,b,c,d,f,g,x,y\}$ and 
$\cale=\{e_1,e_2,e_3,e_4,e_5\}$, 
where $e_1=\{a,b,c\}$, $e_2=\{c,d,y\}$,
$e_3=\{b,c,f,g\}$, $e_4=\{f,g,y\}$
and $e_5=\{x,y,h\}$, then $\calh$ is Zykov-planar, 
as its edges are sets of vertices of some 
faces in the plane graph shown in Figure~\ref{f3}.

\begin{figure}[htbp]
  \centering
\includegraphics[scale=0.6]{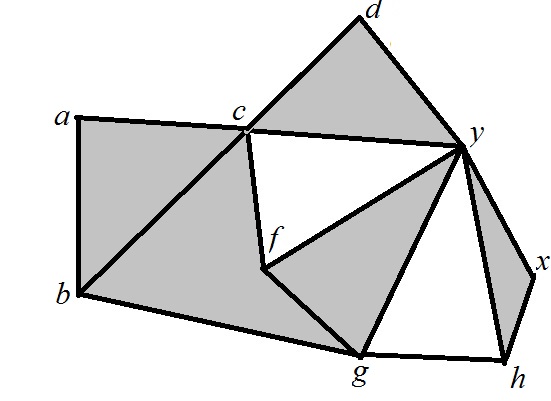}
\caption{A plane graph}

\relabel{f3}
\end{figure}

Johnson and Pollak~\cite{Joh 1987} showed that 
a hypergraph $\calh=(\calv,\cale)$ 
is Zykov-planar if and only if 
the bipartite graph %$H_{\calh}$ 
with vertex set $\calv\cup \cale$ and edge set 
$\{ve: v\in e\in \cale\}$ is planar. 

For any planar graph $G$ without loops, 
Birkhoff and Lewis \cite{birk1946} showed that 
$P(G,\lambda)>0$ holds for all real $\lambda\ge 5$,
implying that $P(G,\lambda)$ has no real roots 
in the interval $[5,\infty)$. 
They also conjectured in the same paper that $P(G,\lambda)>0$ holds for all real $\lambda$ with 
$4\le \lambda<5$.
By the Four-color Theorem~\cite{app, rob} and Johnson and Pollak's 
characterization on Zykov-planar hypergraphs in~\cite{Joh 1987},
every Zykov-planar hypergraph has a weak proper colouring
with 4 colours, i.e.,  
$P(\calh,4)>0$ holds for every Zykov-planar hypergraph $\calh$.
Then it is natural to ask if 
$P(\calh,\lambda)>0$ holds for all real numbers 
$\lambda$ with $\lambda>4$
and all Zykov-planar hypergraphs $\calh=(\calv,\cale)$ with 
$|e|\ge 2$ for all $e\in \cale$.

\iffalse 
extend Birkhoff and Lewis's result and conjecture~\cite{birk1946}
on chromatic polynomials of planar graphs 
to Zykov-planar hypergraphs. 
 
\begin{Conj}\relabel{con1}
For any Zykov-planar hypergraph $\calh=(\calv,\cale)$ with 
$|e|\ge 2$ for all $e\in \cale$,
$P(\calh,\lambda)>0$ holds for all real numbers 
$\lambda$ with $\lambda>4$.
\end{Conj}
\fi

Thomassen~\cite{tho2} proved that 
the real roots of chromatic polynomials of planar graphs 
are dense in the interval $(32/27,3)$.
In the same paper, he conjectured that 
the real roots of chromatic polynomials of 
planar graphs are dense in the interval 
$[3,4)$. %contain a dense subset of the interval from 3 to 4.
Recently Perret and Thomassen~\cite{per 2016}
proved that this conjecture holds 
except for a small interval $(t_1,t_2)$ around 
the number $\frac{5+\sqrt 5}2\approx 3.618033$,
where $t_1\approx 3.618032$ and $t_2 \approx 3.618356$.
Is it true that the real roots of chromatic polynomials 
of Zykov-planar hypergraphs $\calh$ with
$|e|\ge 3$ for some edge $e$ in $\calh$  
are dense in the interval $(32/27, 4)$?

\iffalse 
Now we end this section with the following conjecture. 

\begin{Conj}\relabel{con2}
%For any Zykov-planar hypergraph $\calh=(\calv,\cale)$ with $|e|\ge 2$ for all $e\in \cale$,
The roots of chromatic polynomials 
of Zykov-planar hypergraphs $\calh$ with
$|e|\ge 3$ for some edge $e$ in $\calh$  
are dense in the interval $(32/27, 4)$, 
\end{Conj}
\fi

\vspace{0.5 cm}

\noindent {\bf Acknowledgement.} 
The authors wish to thank the referees 
for their very helpful comments and suggestions.

\end{document}